\documentclass[a4paper,11pt]{amsart}
\usepackage[utf8]{inputenc}
\usepackage{amssymb,amsmath,amsthm}
\usepackage{amsfonts} 
\usepackage{a4wide}
\usepackage{bbm}
\usepackage{bm}
\usepackage[font=small,format=hang,labelfont={sf,bf}]{caption}
\usepackage{color}
\usepackage{empheq}
\usepackage{epsfig}
\usepackage{faktor}
\usepackage{graphicx}
\usepackage{hyperref}
\usepackage{hyphenat}
\usepackage{latexsym}
\usepackage{mathrsfs}
\usepackage{microtype}
\usepackage{orcidlink}
\usepackage{url}
\usepackage{varioref}
\theoremstyle{plain}
\begingroup
\newtheorem{theorem}{Theorem}[section]
\newtheorem{definition}[theorem]{Definition}
\newtheorem{lemma}[theorem]{Lemma}
\newtheorem{proposition}[theorem]{Proposition}

\endgroup
\theoremstyle{definition}
\newtheorem{remark}[theorem]{Remark}

\newcommand{\calA}{\mathcal{A}}

\newcommand{\calE}{\mathrm{E}}

\newcommand{\calH}{\mathcal{H}}

\newcommand{\calP}{\mathcal{P}}

\newcommand{\mathN}{\mathbb{N}}

\newcommand{\mathR}{\mathbb{R}}
\newcommand{\mathS}{\mathbb{S}}
\newcommand{\mathZ}{\mathbb{Z}}
\newcommand{\R}{\mathbb{R}}
\newcommand{\Om}{\Omega}

\newcommand{\diverge}{\mathrm{div}\,}

\newcommand{\diam}{\mathrm{diam}\,}

\newcommand{\supp}{\mathrm{supp}\,}

\newcommand{\Per}{\mathrm{P}}

\definecolor{ddorange}{rgb}{1,0.5,0}
\definecolor{ddcyan}{rgb}{0,0.2,1.0}

\def\avint{\mathop{\,\rlap{---}\!\!\int}\nolimits}

\begin{document}
\title{Optimal sets for a geometric oscillation energy}

\author{Matteo Novaga} 
\address{Universit\`a di Pisa \\ 
Largo Bruno Pontecorvo 5, 56127 Pisa, Italy}
\email[M.~Novaga]{matteo.novaga@unipi.it \orcidlink{0000-0002-0812-3780}}

\author{Fumihiko Onoue} 
\address{Technische Universit\"at M\"unchen, Bolzmannstrasse 3, 85748 Garching, Germany} 
\email[F.~Onoue]{fumihiko.onoue@tum.de \orcidlink{0000-0002-4031-7681}}

\author{Emanuele Paolini} 
\address{Universit\`a di Pisa \\ 
Largo Bruno Pontecorvo 5, 56127 Pisa, Italy}
\email[E.~Paolini]{emanuele.paolini@unipi.it \orcidlink{0000-0001-7021-7537}}

\date{\today}	

\begin{abstract}
We investigate the nonlocal energy corresponding to the $p$-oscillation of the unit normal vector for hypersurfaces, or the unit tangent vector for curves. The energy satisfies geometric inequalities with optimal constants $c(n,p)$ and $C(n,p)$ which are determined by a variational problem over the probability measures on the sphere. The extremal measures for such problem depend critically on the value of $p$. 
We prove existence of optimal sets for this energy under perimeter and volume constraint, and characterize their shape.
\end{abstract}

\maketitle
%%%%%%%%%%%%%%%%%%%%%%%
%\setcounter{tocdepth}{1} 
\tableofcontents
%%%%%%%%%%%%%%%%%%%%%%%%%%%%%%%%%%%%%%
%%%%%%%%%%%%%%%%%%%%%%%%%%%%%%%%%%%%%%
%%%%%%%%%%%%%%%%%%%%%%%%%%%%%%%%%%%%%%

\section{Introduction}
Nonlocal geometric functionals have attracted considerable attention in recent years due to their rich mathematical structure and their appearance in physical models, such as in the theory of membranes and vesicles. A typical example of nonlocal energies is the fractional perimeter, introduced in \cite{CRS10} as a model of phase transitions with long-range interactions. Since then, minimization problems involving the fractional perimeter have been studied by many authors (see for instance \cite{DipVal23, Serra24} and references therein). Another example is the nonlocal Willmore energy or its variants, which involve double integrals of a kernel depending on both position and normal vector (see for instance \cite{BlaRei15, MS19, CeNo20, BRS25} for further details).

In this paper, we consider a nonlocal geometric energy that depends on the normal vectors. More precisely, given $p>0$ and a set $\Om\subset \mathR^n$ of finite volume and finite perimeter, we consider the energy 
\begin{equation}\label{defEp}
    \mathrm{E}_p(\Omega) \coloneqq  \iint_{\partial^*\Omega\times\partial^*\Omega}
    |\nu_\Omega(x)-\nu_\Omega(y)|^p\,d\mathcal{H}^{n-1}(x)\,d\mathcal{H}^{n-1}(y)
\end{equation}
where $\partial^*\Omega$ denotes the reduced boundary of $\Omega$ and $\nu_\Omega(x)$ the measure-theoretic outer unit normal at $x \in \partial^*\Omega$ (see \cite{AFP00}). 
% Note that if $\partial \Omega$ is sufficiently smooth, then we often write $\calE_p(\Sigma) = \mathrm{E}_p(\Omega)$. 
Notice that, if we consider a rescaled set $\lambda\Om$ with $\lambda>0$, we have
\[
\mathrm{E}_p(\lambda\Omega)=\lambda^{2n-2}\mathrm{E}_p(\Omega).
\]
A natural question is understanding the range of possible values of $\calE_p(\Om)$ for sets $\Om$ 
with a prescribed perimeter or volume. In particular, we shall prove that there exist optimal constants $c(n,p)$ and $C(n,p)$, depending only on $n$ and $p$, such that
\begin{equation}\label{mainineq}
    c(n,p)\,\Per(\Om)^2 \le \mathrm{E}_p(\Om) \le C(n,p)\,\Per(\Om)^2
\end{equation}
for every set $\Om \subset \mathR^n$ of finite perimeter. 
For some values of $p$, we can also characterize the 
optimal sets realizing the equalities in \eqref{mainineq}, that is, we consider the variational problems
\begin{eqnarray}
\label{minEp}
\min\left\{\mathrm{E}_p(\Om):\ \Om\subset\R^n,\ |\Om|<+\infty,
\ \Per(\Om)=P_0\in (0,+\infty)\right\},
\\
\label{maxEp}
\max\left\{\mathrm{E}_p(\Om):\ \Om\subset\R^n,\ |\Om|<+\infty,
\ \Per(\Om)=P_0\in (0,+\infty)\right\}.
\end{eqnarray}
These problems reduce to the study of a quadratic functional on probability measures on the unit sphere $\mathS^{n-1}$. Indeed, letting $\mu_\Om\in\mathcal P(\mathS^{n-1})$ be the push-forward of the normalized area measure $\Per(\Om)^{-1}\,\mathcal H^{n-1}\lfloor_{\partial^*\Om}$ via the Gauss map $\nu_\Om$ (see Section \ref{sec:preliminaries} for the precise definition), we have
\[
    \mathrm{E}_p(\Om)=\Per(\Om)^2 J_p(\mu_\Om), \qquad \textrm{where}\quad
    J_p(\mu):=\iint_{\mathS^{n-1}\times\mathS^{n-1}}|v-w|^p\,d\mu(v)d\mu(w),
\]
with the constraint $\int_{\mathS^{n-1}} v\,d\mu_\Omega(v)=0$. Thus, the constants in \eqref{mainineq} are given by
\begin{equation}\label{defconstants}
    c(n,p)=\min_{\mu\in\mathcal{P}_0(\mathS^{n-1})} J_p(\mu), \qquad 
    C(n,p)=\max_{\mu\in\mathcal{P}_0(\mathS^{n-1})} J_p(\mu),
\end{equation}
where $\mathcal{P}_0(\mathS^{n-1})$ denotes the set of probability measures on $\mathS^{n-1}$ with barycenter in the origin.

\smallskip
The paper is organized as follows: in Section \ref{sec:preliminaries} we set up the notation and recall some basic facts about sets of finite perimeter and Radon measures. In Section \ref{sec:constants} we give some estimates and, in some cases, we characterize
the constants $c(n,p)$ and $C(n,p)$ for different ranges of $p$. 
%We also analyze the limit $p\to\infty$. 
In Section \ref{sec:perimeter} we characterize the optimal sets for problems \eqref{minEp} and \eqref{maxEp}, and in Section \ref{sec:volume} we consider the corresponding minimum problem with a volume constraint. Finally, in Section \ref{seccurves} we discuss analogous problems for closed curves in $\R^n$ of fixed length. 

\smallskip
\textsc{Acknowledgments.} The authors wish to thank Marco Pozzetta 
for useful discussions on this problem. 
M.N. was partially supported by Next Generation EU, 
PRIN 2022E9CF89 and PRIN PNRR P2022WJW9H; 
E.P. was partially supported by Next Generation EU, PRIN 2022PJ9EFL.
M.N and E.P. are members of INDAM-GNAMPA.

\section{Notation}\label{sec:preliminaries}
Let $n \geq 2$. Given a measurable set $\Omega\subset\mathR^n$, its perimeter $\Per(\Omega)$ is defined as
\[
    \Per(\Omega)=\sup\left\{\int_{\mathR^n} \chi_{\Omega}\, \diverge g\,dx :\ g\in C^1_c(\mathR^n;\mathR^n),\, \|g\|_\infty\le 1\right\}
\]
where $\chi_{\Omega}$ is the characteristic function of $\Omega$. If $\Per(\Omega)<+\infty$, we say that $\Omega$ is a set of finite perimeter. In this case, there exists a $\calH^{n-1}$-rectifiable set $\partial^*\Omega$, called reduced boundary, and a measurable function $\nu_\Omega:\partial^*\Omega\to\mathS^{n-1}$, called measure-theoretic outer unit normal.
The energy $\mathrm{E}_p$ in \eqref{defEp} is well-defined for any set $\Omega$ of finite perimeter, since the normal map $\nu_\Omega$ is defined $\calH^{n-1}$-almost everywhere on $\partial^*\Omega$.
Gauss-Green Theorem ensures that, if $\Omega$ has finite measure, then $\int_{\partial^*\Omega} \nu_\Omega(x)\, d\calH^{n-1}(x)=0$.

We denote by $\mathcal{P}(\mathS^{n-1})$ the set of probability measures on the unit sphere $\mathS^{n-1}\subset\mathR^n$, and by $\mathcal{P}_0(\mathS^{n-1})$ the subset of 
$\mathcal{P}(\mathS^{n-1})$ with barycenter in the origin.  
We recall that, for $\mu\in\mathcal{P}(\mathS^{n-1})$, 
the barycenter of $\mu$ is the vector
\[
    \overline{\mu}=\int_{\mathS^{n-1}} v\,d\mu(v)\in\mathR^n.
\]
Given a set $\Om$ of finite volume and finite perimeter, we define the push-forward of the area measure $\calH^{n-1}$ of $\partial^* \Omega$ with respect to its measure-theoretic outer unit normal by
\[
\mu_\Om:=\frac{(\nu_\Om)_\#\calH^{n-1} \lfloor_{\partial^*\Om}}
{\Per(\Om)}\in\mathcal{P}_0(\mathS^{n-1}).
\]
Notice that a measure $\mu_\Om$ corresponding to a set $\Om$ of finite perimeter cannot be supported on a closed hemisphere. 

\section{A variational problem on probability measures}

For $p>0$ and $\mu\in\mathcal{P}(\mathS^{n-1})$, we define
\begin{equation}\label{defJp}
    J_p(\mu):=\iint_{\mathS^{n-1}\times\mathS^{n-1}}|v-w|^p\,d\mu(v)d\mu(w).
\end{equation}
Since the kernel $|v-w|^p$ is continuous on $\mathS^{n-1}\times\mathS^{n-1}$, the functional $J_p$ is continuous with respect to the weak-star convergence of measures.

Two particular measures will play a special role: 
the uniform measure $\sigma_{n-1}$ on $\mathS^{n-1}$, defined as 
$\sigma_{n-1}:=\calH^{n-1}(\mathS^{n-1})^{-1}\calH^{n-1}\lfloor_{\mathS^{n-1}}$; 
the measure $\mu_{\rm sim}$, defined up to rotations, uniformly distributed on the vertices of a regular $n$-simplex inscribed in $\mathS^{n-1}$. More precisely, letting $v_0,\dots,v_n\in\mathS^{n-1}$ be such that $|v_i-v_j|$ is constant for $i\neq j$, 
then $\mu_{\rm sim}=\frac{1}{n+1}\sum_{i=0}^n\delta_{v_i}$. 
Explicit computations yield the values
\begin{align*}
    &J_p(\sigma_{n-1})=\frac{2^{p+n-2}\Gamma\bigl(\frac{n}{2}\bigr)
    \Gamma\bigl(\frac{p+n-1}{2}\bigr)}{\sqrt\pi\,
    \Gamma\bigl(\frac{p}{2}+n-1\bigr)}, \\
    &J_p(\mu_{\rm sim})=2^{\frac{p}{2}}\Bigl(1+\frac1n\Bigr)^{\frac{p}{2}-1}.
\end{align*}

The following result, which follows directly from the definition of push-forward measure, shows that the energy $\calE_p(\Om)$ can be expressed in terms of the measure $\mu_\Om$.

\begin{proposition}\label{prop:measure-reformulation}
    Let $\Omega\subset\mathR^n$ be a set of finite perimeter. Then
    \[
    \calE_p(\Om) = \Per(\Om)^2\,J_p(\mu_\Om).
    \]
\end{proposition}

The following result provides a characterization of such measures (see \cite[Theorems 7.2.1 and 8.2.2]{S14}).

\begin{theorem}[Minkowski Theorem]\label{teomink}
Let $\mu\in\mathcal{P}_0(\mathS^{n-1})$ be such that the support of $\mu$ is not contained in a closed hemisphere. Then there exists a bounded convex set $C \subset \mathR^n$ such that $\Per(C)=1$ and $\mu=\mu_C$, i.e., 
\begin{equation*}
    \mu(S) = \calH^{n-1}( \{ x \in \partial^* C \, : \, \nu_C(x) \in S \})
\end{equation*}
for any Borel subset $S \subset \mathS^{n-1}$. Moreover the convex set $C$ is unique up to translations.
\end{theorem}

\begin{remark}\label{rmk:applicationMinkowski}
As an application of Theorem \ref{teomink} we obtain that, given $\Omega \subset \mathR^n$ of finite perimeter, there exists a bounded convex set $C \subset \mathR^n$, unique up to translations, such that
\begin{equation*}
\mu_\Om = \mu_C, \qquad
    \Per(\Omega) = \Per(C), \qquad \mathrm{E}_p(\Omega) = \mathrm{E}_p(C).
\end{equation*}
Indeed, by applying the Minkowski theorem to the measure $\mu_\Om$,
%$\widetilde{\mu}_{\Omega} \coloneqq (\nu_{\Omega})_{\#} \calH^{n-1}\lfloor_{\partial^* \Omega}$ on $\mathS^{n-1}$, 
we find a bounded convex set $\widehat C \subset \mathR^n$ such that $\mu_{\Omega} = \mu_{\widehat C}$ and $\Per(\widehat{C}) =1$. Setting now $C=\Per(\Om)^\frac{1}{n-1}\, \widehat{C}$, we have $\Per(\Omega) = \Per(C)$ and
%\begin{equation*}
%    \Per(\Omega) = \int_{\partial^* \Omega}|\nu_{\Omega}| \,d\calH^{n-1} = \int_{\mathS^{n-1}} |v| \, d\widetilde{\mu}_{\Omega}(v) =  \int_{\mathS^{n-1}} |v| \, d\widetilde{\mu}_{C}(v) = \Per(C). 
%\end{equation*}
%Moreover, by definition, we have 
\begin{equation*}
    \mathrm{E}_p(\Omega) = \Per(\Om)^2\,J_p(\mu_{\Omega}) 
    =  \Per(C)^2\,J_p(\mu_{C}) = \mathrm{E}_p(C).
\end{equation*}
\end{remark}

\begin{proposition}\label{provol}
Let $\Omega \subset \mathbb{R}^n$ be a bounded set of finite perimeter and let $C \subset \mathbb{R}^n$ be the convex body in Remark \ref{teomink}, 
%and let $C_\Om=\Per(\Om)^\frac{1}{n-1} C$, 
such that $\Per(C)=\Per(\Om)$ and $\mathrm{E}_p(\Omega) =\mathrm{E}_p(C)$.
Then, we have
\begin{equation}\label{com}
|C|\ge |\Omega|,
\end{equation}
and equality holds if and only if $\Omega=C$, up to translations and up to a set of measure zero.
\end{proposition}

\begin{proof}
Let $h_C : \mathS^{n-1} \to \mathbb{R}$ be the support function of the set $C$, defined as 
\[
h_C(u) = \sup_{y \in C} y \cdot u.
\]
A classical result in convex geometry (see \cite[Remark 5.1.2]{S14}) expresses the volume of  $C$ in terms of its support function and its surface area:
\begin{align} \label{eq:mixed_vol}
|C| &= \frac{\Per(C)}{n} \int_{\mathS^{n-1}} h_C(u) \, d\mu_C(u) \nonumber\\
&= \frac{\Per(\Om)}{n} \int_{\mathS^{n-1}} h_C(u) \, d\mu_\Omega(u) \nonumber\\
&=  \frac{1}{n} \int_{\partial^* \Omega} h_C(\nu_\Omega(x)) \, d\mathcal{H}^{n-1}(x) = \frac{1}{n}\, \Per_C(\Omega),
\end{align}
where we denote by $P_C(\Omega)$ the anisotropic perimeter of $\Omega$ with respect to the surface tension $h_C$ (with this notation, $C$ corresponds to the Wulff shape of $\Per_C$). See \cite[Proposition 2.6]{FM91} for this correspondence. Here we have used that $\mu_\Om=\mu_C$ and $\Per(\Omega)= \Per(C)$.

%Assuming that $C$ contains the origin in its interior,
% Setting $\widetilde{\Omega} \coloneqq (\frac{|C|}{|\Omega|})^{\frac{1}{n}} \, \Omega$, we have $|\widetilde{\Omega}| = |C|$. Then, 
By Wulff's inequality (see for instance \cite{Busemann49, FM91}), we obtain that
\begin{equation}\label{wulffInequality}
    \Per_C(\Omega) \geq n \, |C|^{\frac{1}{n}} \, |\Omega|^{\frac{n-1}{n}}
\end{equation}
and the equality holds if and only if $\Omega$ is homothetic to $C$. From \eqref{eq:mixed_vol}, we know that $\Per_C(\Omega) = n\, |C|$ and thus, \eqref{wulffInequality}, we obtain
\[
    n \, |C| = \Per_C(\Omega) \geq  n \, |C|^{\frac{1}{n}} \, |\Omega|^{\frac{n-1}{n}},
\]
which implies that 
\[
    |C| \geq  |\Omega|.
\]
Here we have assumed $|C| > 0$; otherwise $\Omega$ is negligible and the inequality is trivial.
%, we divide both sides by $n |C|^{\frac{1}{n}}$ to get
%\[
%|C|^{\frac{n-1}{n}} \geq |\Omega|^{\frac{n-1}{n}},
%\]

If $|C| = |\Omega| $, then
%the anisotropic isoperimetric inequality holds as an equality:
\[
    \Per_C(\Omega)  =  \Per_C(C),
\]
so that $\Omega$ must be equal to $C$, up to translations and up to negligible sets.
\end{proof}

As shown in \cite[Theorem 4.2.1]{S14}, the weak-star convergence of probability measures is equivalent to the convergence of the corresponding convex bodies.

\begin{theorem}\label{teominkconv}
Let $\mu \in \calP_0(\mathS^{n-1})$ and let $(\mu_n)_{n \in \mathN} \subset \mathcal{P}_0(\mathS^{n-1})$. Then,
$\mu_n \rightarrow  \mu$ in the weak-star topology if and only if $\chi_{C_n} \to \chi_{C}$ in $L^1$ where $C_n$ (respectively $C$) are the convex sets corresponding to $\mu_n$ (respectively $\mu$) as in Theorem \ref{teomink}.
\end{theorem}

\begin{definition}\label{def:constants}
    For $n\ge 2$ and $p>0$, we define the constants
    \[
        c(n,p):=\inf_{\mu\in\mathcal{P}_0(\mathS^{n-1})} J_p(\mu), \qquad 
        C(n,p):=\sup_{\mu\in\mathcal{P}_0(\mathS^{n-1})} J_p(\mu).
    \]
\end{definition}
By the compactness of $\mathcal{P}_0(\mathS^{n-1})$ and the continuity of $J_p$ with respect to the weak-star convergence of measures, the infimum and the supremum are actually attained. With this definitions, inequality \eqref{mainineq} follows immediately.

\begin{theorem}\label{thm:mainineq}
    For any set $\Omega$ of finite perimeter, we have
    \[
        c(n,p)\,\Per(\Om)^2 \le \calE_p(\Om) \le C(n,p)\,\Per(\Om)^2.
    \]
    Moreover, for any $\varepsilon>0$ there exist 
    $\Om^-_\varepsilon,\Om^+_\varepsilon$ such that
    \[
        \calE_p(\Om^-_\varepsilon) \le (c(n,p)+\varepsilon)\,\Per(\Om^-_\varepsilon)^2, \qquad 
        \calE_p(\Om^+_\varepsilon) \ge (C(n,p)-\varepsilon)\,\Per(\Om^+_\varepsilon)^2.
    \]
\end{theorem}

\begin{remark}
The results in this section still hold for functionals of the form
\[
\calE_f(\Om):=\iint_{\partial^*\Omega\times\partial^*\Omega}
f(|\nu_\Omega(x)-\nu_\Omega(y)|)\,d\mathcal{H}^{n-1}(x)\,d\mathcal{H}^{n-1}(y),
\]
where $f:[0,2]\to\R$ is a given continuous function.
\end{remark}

\section{Estimates for the optimal values}\label{sec:constants}

\begin{proposition}\label{prop:basic}
    For all $p>0$ we have
    \[ 
    \min\{2^{\frac{p}{2}},\,2^{p-1}\} \le c(n,p) \le C(n,p) \le 2^{p}.
    \]
\end{proposition}
\begin{proof}
    For all $v,w\in \mathS^{n-1}$ we have $|v-w|^p = 2^{\frac{p}{2}}(1-v\cdot w)^{\frac{p}{2}}$.
    For any $\mu\in\mathcal{P}_0(\mathS^{n-1})$, 
    and $p\ge 2$, by Jensen's inequality we have
    \[
        \iint (1-v\cdot w)^{\frac{p}{2}} \, d\mu(v)\, d\mu(w) \ge \left(\iint (1-v\cdot w)\, d\mu(v)\, d\mu(w)\right)^{\frac{p}{2}}=1^{\frac{p}{2}}=1,
    \]
    since $\iint v\cdot w\,d\mu(v)\,d\mu(w)=|\overline{\mu}|^2=0$. Thus $J_p(\mu)\ge 2^{\frac{p}{2}}$.\\ 
    For $p<2$ we will see in Proposition \ref{thm:p<2-min} below that $c(n,p)=2^{p-1}<2^{\frac{p}{2}}$.

    The upper bound $C(n,p)\le 2^{p}$ is trivial since $|v-w|\le 2$ for all $v,w$. 
    %The fact that $J_p(\sigma_{n-1})$ lies between $c(n,p)$ and $C(n,p)$ is obvious.
\end{proof}

We now consider different ranges of $p$.

%\subsection{The case $p=2$}\label{sec:p=2}

\begin{proposition}\label{prop:p=2}
    For $p=2$, we have
    \[
        J_2(\mu)=2
    \]
    for any $\mu\in\mathcal{P}_0(\mathS^{n-1})$. Consequently, $c(n,2)=C(n,2)=2$.
\end{proposition}
\begin{proof}
    Compute:
    \[
        |v-w|^2 = 2(1-v\cdot w).
    \]
    Integrating with respect to $\mu\otimes\mu$ gives
    \[
        J_2(\mu)=2\left(1-\int v\,d\mu(v)\cdot\int w\,d\mu(w)\right)=2.
    \]
\end{proof}
%In particular, for $p=2$ the energy $\calE_2(\Om)$ is equal to $2\,\Per(\Om)^2$.
The following Lemma is taken from \cite[Lemma 1]{Bjorck56} (see also \cite{Frostman50, BHS19}).
\begin{lemma}\label{lm:bjork}
Let $0<p<2$ and let $\nu$ be a real measure in $\mathS^{n-1}$ with $\nu(\mathS^{n-1})=0$. Then $J_p(\nu)\le 0$ and the equality holds if and only if $\nu=0$.
\end{lemma}

\begin{proposition}\label{thm:p<2-min}
    For $0<p<2$, we have
    \[
        c(n,p)=2^{p-1},
    \]
    and the minimum is attained only by measures of the form $\mu=\frac12(\delta_v+\delta_{-v})$ for some $v\in\mathS^{n-1}$. Moreover,
    \[
        C(n,p)=J_p(\sigma_{n-1}).
        %= \frac{2^{p+n-1}\Gamma\bigl(\frac{n+1}{2}\bigr)
        %\Gamma\bigl(\frac{p+n}{2}\bigr)}{\sqrt\pi\,
        %\Gamma\bigl(\frac{p+2n}{2}\bigr)}.
    \]
\end{proposition}
\begin{proof}
    For any $\mu\in\mathcal{P}_0(\mathS^{n-1})$, write $|v-w|^p = 2^p \left(\frac{|v-w|}{2}\right)^p$. 
    %Since the function $t\mapsto t^p$ is concave on $[0,\infty)$ for $0<p<1$ and convex for $1<p<2$, we need a uniform bound. 
    Consider the inequality
    \[
        t^p \ge 2^{p-2} t^2 \quad \text{for } t\in[0,2],
    \]
    which follows from the fact that the function $t\mapsto t^p/t^2 = t^{p-2}$ is decreasing on $[0,2]$, and equals $2^{p-2}$ at $t=2$. Thus
    \[
        |v-w|^p \ge 2^{p-2} |v-w|^2.
    \]
    Integrating both sides of the above inequality with respect to $(v,w) \in \mathS^{n-1} \times \mathS^{n-1}$ gives
    \[
        J_p(\mu) \ge 2^{p-2} J_2(\mu) = 2^{p-2}\cdot 2 = 2^{p-1},
    \]
    since $J_2(\mu)=2$ for any $\mu\in\mathcal{P}_0(\mathS^{n-1})$. Equality holds if and only if $|v-w|$ is either 0 or 2 $(\mu\otimes\mu)$-almost everywhere, i.e., $\mu$ is supported on two antipodal points. Such a measure has zero barycenter only if the two points are opposite and carry equal mass $\frac{1}{2}$. This proves the statement for $c(n,p)$.

    Regarding the maximum, we claim that the uniform measure $\sigma_{n-1}$ maximizes $J_p$ when $0 < p < 2$. We first show that the energy $J_p$ is concave in $\calP_0(\mathS^{n-1})$. Indeed, let $\mu_1, \, \mu_2 \in \calP_0(\mathS^{n-1})$ and set $\nu \coloneqq \mu_1 - \mu_2$. By definition, we have that $\nu(\mathS^{n-1}) = 0$. By Lemma~\ref{lm:bjork}, 
    we obtain
    \begin{equation*}
        J_p(\nu) \leq 0. 
    \end{equation*}
    We then get 
    \begin{align}
        J_p(t \, \mu_1 + (1-t)\,\mu_2) &= t^2 \, J_p(\mu_1) + (1-t)^2 \, J_p(\mu_2) \nonumber\\
        &\qquad + 2\,t\,(1-t)\,\iint_{\mathS^{n-1} \times \mathS^{n-1}} | v - w|^p \,d\mu_1(v) \, d\mu_2(w) \nonumber\\
        &= t^2 \, J_p(\mu_1) + (1-t)^2 \, J_p(\mu_2) + t\,(1-t)\, \left( J_p(\mu_1) + J_p(\mu_2) - J_p(\nu) \right) \nonumber\\
        &= t\, J_p(\mu_1) + (1-t) \, J_p(\mu_2) - t\,(1-t) \, J_p(\nu) \nonumber\\
        &\geq t\, J_p(\mu_1) + (1-t) \, J_p(\mu_2)
    \end{align}
    for any $t \in [0,\,1]$, which proves the concavity of $J_p$. Since the set $\calP_0(\mathS^{n-1})$ is convex and $J_p$ is continuous, there exists a unique maximizer which is necessarily the uniform measure $\sigma_{n-1}$. 
    \end{proof}

    \begin{remark}
        Notice that the measure $\mu=\frac12(\delta_v+\delta_{-v})$ cannot be realized as $\mu_\Om$ for a set $\Om$ of finite perimeter, because this forces the set to be an unbounded slab, which has infinite perimeter. Nevertheless, the constant $c(n,p)$ is optimal in the sense that there exist sequences of sets $\Om_k$ with $\Per(\Om_k)=1$ for all $k$ and $\calE_p(\Om_k)\to 2^{p-1}$ as $k\to +\infty$. 
        %See Theorem \ref{thm:p<2-per} for the construction of such sets $\Omega_k$.
    \end{remark}

%\subsection{The case $2<p<4$}\label{sec:2<p<4}

For $p>2$, antipodal measures are the unique maximizers.

\begin{proposition}\label{max:2<p}
For $p>2$ we have
    \[
        C(n,p)=2^{p-1},
    \]
    and the maximum is attained only by measures of the form $\mu=\frac12(\delta_v+\delta_{-v})$ for $v \in \mathS^{n-1}$.
\end{proposition}

\begin{proof}
The result is contained in \cite[Theorem 4.6.6]{BHS19} for the proof; however we repeat the proof here for convenience. We first notice that
    \[
        |v-w|^p \le 2^{p-2}|v-w|^2, 
    \]
    which follows from $t^p\le 2^{p-2}t^2$ for $t\in[0,2]$, since $t^{p-2}$ is strictly increasing
    for $p>2$. Thus,
    \[
        J_p(\mu)\le 2^{p-2} J_2(\mu)=2^{p-1},
    \]
    with equality if and only if $|v-w|$ is equal to either $0$ or $2$ for $(\mu\otimes\mu)$-a.e. $(v,\,w) \in \mathS^{n-1} \times \mathS^{n-1}$, i.e. $\mu$ is supported on two antipodal points, hence $C(n,p)=2^{p-1}$.
\end{proof}

For $2<p<4$, the uniform measure $\sigma_{n-1}$ becomes the unique minimizer, while the maximum is still achieved by the antipodal measure.

\begin{proposition}\label{thm:2<p<4}
    For $2<p<4$ we have
    $c(n,p)=J_p(\sigma_{n-1})$, 
        %= \frac{2^{p+n-1}\Gamma\bigl(\frac{n+1}{2}\bigr)
        %\Gamma\bigl(\frac{p+n}{2}\bigr)}{\sqrt\pi\,
        %\Gamma\bigl(\frac{p+2n}{2}\bigr)},
    and $\sigma_{n-1}$ is the unique minimizer.
\end{proposition}
\begin{proof}
    The proof is similar to that of Proposition \ref{thm:p<2-min}. 
    In Lemma \ref{lem:positivity} below we show that, for $2<p<4$, we have
    \begin{equation}\label{eqnu}
        J_p(\nu)\ge 0
    \end{equation}
    for any signed measure $\nu$ with 
    \begin{equation}\label{condnu}
    \nu(\mathS^{n-1})=0 \quad  \text{and} \quad \int_{\mathbb{S}^{n-1}} v\,d\nu(v)=0. 
    \end{equation}
    The equality holds if and only if $\nu=0$.   

%To show \eqref{eqnu} we expand the interaction energy in terms of spherical harmonics. 
%Let $\{Y_{k,j}\}_{j=1}^{d_k}$ denote an orthonormal basis of spherical harmonics on $\mathbb{S}^{n-1}$ of degree $k$, where $d_k$ is the dimension of the corresponding eigenspace.
%Since the interaction kernel $|x-y|^p = (2 - 2x\cdot y)^{p/2}$ depends only on the inner product $x \cdot y$, the functional $J_p(\nu)$ takes the form
%\begin{equation}\label{eq:spectral_decomp}
%J_p(\nu) = \sum_{k=0}^\infty \lambda_k(n,p) \sum_{j=1}^{d_k} |\hat{\nu}_{k,j}|^2
%\quad \text{with } \hat{\nu}_{k,j} = \int_{\mathbb{S}^{n-1}} \overline{Y_{k,j}(x)} \, d\nu(x).
%\end{equation}
%The eigenvalues $\lambda_k(n,p)$ correspond to the coefficients in the Legendre expansion of the function $t \mapsto (2(1-t))^{p/2}$, and are explicitly given by
%\begin{equation}\label{eq:eigenvalues_gamma}
%\lambda_k(n,p) = 2^{p}\, \pi^{n/2} 
%\frac{\Gamma\left(\frac{p+n}{2}\right)\,\Gamma\left(k - \frac p2\right)}
%{\Gamma\left(-\frac p2\right)\,\Gamma\left(k + \frac{p+n}{2}\right)}\,.
%\end{equation}
%The conditions \eqref{condnu} imply that $\hat{\nu}_{0,1} = 0$ and $\hat{\nu}_{1,j} = 0$ for all $j=1,\dots,d_1$, so that the sum in \eqref{eq:spectral_decomp} runs over $k \ge 2$.
%Since $2<p<4$, all the terms in \eqref{eq:eigenvalues_gamma} involving the function $\Gamma$ are strictly positive, so that $\lambda_k(n,p) > 0$ for all $k \ge 2$,
%which implies \eqref{eqnu}.
   Let $\mu_1,\,\mu_2\in\mathcal{P}_0(\mathS^{n-1})$ and set $\nu \coloneqq \mu_1 - \mu_2$. For the same reason as in the proof of Proposition \ref{thm:p<2-min}, we have
    \begin{align*}
        J_p\left(t\,\mu_1 + (1-t) \, \mu_2 \right) &= t\, J_p(\mu_1)
        + (1-t) J_p(\mu_2) - t\,(1-t)\,  J_p(\nu) \\
        &\leq t\, J_p(\mu_1)+ (1-t) \,J_p(\mu_2)
    \end{align*}
    for any $t \in [0,\,1]$. Here we have used \eqref{eqnu} in the last inequality. Thus, it follows that $J_p$ is strictly convex on $\mathcal{P}_0(\mathS^{n-1})$. Since the set $\mathcal{P}_0(\mathS^{n-1})$ is convex and $J_p$ is continuous, there exists a unique minimizer which is necessarily the uniform measure $\sigma_{n-1}$.
\end{proof}

\begin{lemma}\label{lem:positivity}
For $2<p<4$, the inequality    
\eqref{eqnu} holds for any signed measure $\nu$ on $\mathbb S^{n-1}$ 
satisfying the moment condition \eqref{condnu}. Moreover, the equality in \eqref{eqnu} holds if and only if $\nu =0$.
\end{lemma}

\begin{proof}
The argument is based on the representation of $J_p(\nu)$ by Fourier transform.

We first observe that the polynomial $|v-w|^p$, for $2 < p < 4$, can be expressed as %absolutely convergent integral as 
% For $0<p<2$ the Fourier transform of $|x|^{p}$, understood in the sense of tempered distributions, can be expressed (see \cite[Chapter I]{Stein70}) through the absolutely convergent integral representation
% \begin{equation}\label{eq:rep02}
% |v-w|^{p}
% = C_{n,p}
%    \int_{\mathbb{R}^{n}}
%    \frac{1-\cos\bigl(\xi\cdot(v-w)\bigr)}
%         {|\xi|^{n+p}}\,d\xi ,
% \qquad (0<p<2),
% \end{equation}
% for $v, \, w \in \mathS^{n-1}$ where the positive constant $C_{n,p}$ is given explicitly by
% \begin{equation}\label{eq:const02}
% C_{n,p}
% = \frac{2^{p}\,\pi^{n/2}\,\Gamma\bigl(\frac{p}{2}\bigr)}
%         {\Gamma\bigl(\frac{n-p}{2}\bigr)} .
% \end{equation}
% See \cite[Proposition 3.3]{DiNPV12} for the computation.
%The integral in \eqref{eq:rep02} converges absolutely because the numerator behaves like $|\xi|^2$ for small $|\xi|$, compensating the singularity $|\xi|^{-n-p}$ (since $p<2$), while for large $|\xi|$ the decay of $|\xi|^{-n-p}$ ensures integrability.

% For $2<p<4$ a direct analogue of \eqref{eq:rep02} would diverge at $\xi=0$ because the numerator $1-\cos(\xi\cdot x)$ is only of order $|\xi|^2$.  
% To obtain a convergent representation one subtracts the quadratic term of the Taylor expansion, which is justified under the moment conditions \eqref{eq:moments},  
% and the correct formula becomes
% %
\begin{equation}\label{eq:rep}
    |v-w|^{p} = K(n,p) \, \int_{\mathbb{R}^{n}} \frac{1-\cos\bigl(\xi\cdot(v-w)\bigr)-\frac12\bigl(\xi\cdot(v-w)\bigr)^{2}}{|\xi|^{n+p}}\,d\xi,
\end{equation}
with a constant $K(n,p)<0$. 
%\begin{equation}\label{eq:const}
%    C(n,p) = - \frac{2^{p-1}\Gamma\bigl(\frac{p+1}{2}\bigr)}{\pi^{\frac{n}{2}}\,\Gamma\bigl(-\frac{p}{2}\bigr)}.
%\end{equation}
%\fumihiko{I'm not sure the above expression of $C(n,p)$ is still true, but it is enough to have the negativity of $C(n,p)$ as follows.}

To check this, we first see that
\begin{equation*}
    1-\cos\bigl(\xi\cdot z\bigr)-\frac12\bigl(\xi\cdot z \bigr)^{2} = \mathrm{O}(|\xi|^4)
\end{equation*}
for any $\xi \in \mathR^n$ sufficiently small and any $z \in \mathR^n$. Thus, we have
\begin{equation*}
    \int_{B_{\varepsilon}(0)} \frac{1-\cos\bigl(\xi\cdot z \bigr)-\frac12\bigl(\xi\cdot z\bigr)^{2}}{|\xi|^{n+p}}\,d\xi \leq C' \, \int_{B_{\varepsilon}(0)} \frac{d\xi}{|\xi|^{n+p-4}} < + \infty
\end{equation*}
for $\varepsilon > 0$ sufficiently small where $C'$ is a positive constant since $p < 4$. On the other hand, we have
\begin{equation*}
    \int_{B^c_{\varepsilon}(0)} \frac{1-\cos\bigl(\xi\cdot z \bigr)-\frac12\bigl(\xi\cdot z \bigr)^{2}}{|\xi|^{n+p}}\,d\xi \leq C''\, \int_{B^c_{\varepsilon}(0)} \frac{d\xi}{|\xi|^{n+p-2}} < +\infty
\end{equation*}
for any $\varepsilon > 0$ where $C''$ is a positive constant since $p>2$. Therefore, the right-hand side of \eqref{eq:rep} is absolutely convergent for any $v,\,w \in \mathS^{n-1}$ when $2 < p < 4$. Moreover we notice, by using the rotation map $R(|\xi|\,e_1) = \xi$ for any $\xi \in \mathR^n$ where $e_1 = (1,\,0,\cdots,0) \in \mathR^n$ and the change of variables, that
\begin{align*}
    &\int_{\mathbb{R}^{n}} \frac{1-\cos\bigl(\xi\cdot(v-w)\bigr)-\frac12\bigl(\xi\cdot(v-w)\bigr)^{2}}{|\xi|^{n+p}}\,d\xi \nonumber\\
    &= \int_{\mathR^n} \frac{1-\cos(R^{-1}(\eta) \cdot (v-w)) - \frac{1}{2}\,(R^{-1}(\eta) \cdot (v-w))^2}{|\eta|^{n+p}} \,d\eta \nonumber\\
    &= \int_{\mathR^n} \frac{1-\cos(\eta \cdot (|v-w|\,e_1)) - \frac{1}{2}\,(\eta \cdot (|v-w|\,e_1))^2}{|\eta|^{n+p}} \,d\eta \nonumber\\
    &= |v-w|^p\, \int_{\mathR^n} \frac{1-\cos(\xi \cdot e_1) - \frac{1}{2}\,(\xi \cdot e_1)^2}{|\xi|^{n+p}} \,d\xi \nonumber\\
    &\eqqcolon K(n,p) \, |v-w|^p
\end{align*}
for any $v, \,w \in \mathS^{n-1}$. Now we compute the constant $C(n,p)$ and show that $C(n,p)$ is negative. By the Fubini-Tonelli theorem, we have
\begin{align*}
    K(n,p) &= \int_{\mathR^{n-1}} \int_{\mathR} \frac{1- \cos(\xi_1) - \frac{1}{2}\xi_1^2}{|\xi_1|^{n+p}\, \left( 1 + \frac{|\xi'|^2}{\xi_1^2} \right)^{\frac{n+p}{2}}} \,d\xi_1 \,d\xi' \nonumber\\
    &= \int_{\mathR^{n-1}} \frac{d\xi'}{(1+|\xi'|^2)^{\frac{n+p}{2}}} \, \int_{\mathR} \frac{1 - \cos x - \frac{1}{2}x^2}{|x|^{1+p}}\,dx. 
    % &= |\mathS^{n-2}| \, \int_{0}^{\infty} \frac{r^{n-2}}{(1+r^2)^{\frac{n+p}{2}}}\,dr \, \int_{\mathR} \frac{1 - \cos x - \frac{1}{2}x^2}{|x|^{1+p}}\,dx. 
\end{align*}
Since $2 < p < 4$, by using the integration-by-part formula, we obtain
\begin{equation*}
    \int_{\mathR} \frac{1 - \cos x - \frac{1}{2}x^2}{|x|^{1+p}}\,dx = \frac{2}{p\, (p-1)}\, \int_{0}^{\infty} \frac{\cos x - 1}{x^{p-1}}\,dx = \frac{-2}{p\,(p-1)}\,\int_{0}^{\infty} \frac{1-\cos x}{x^{p-1}}\,dx < 0,
\end{equation*}
and thus $K(n,p) < 0$.
% Note that the subtraction of the quadratic term makes the integral absolutely convergent for every $v,\, w \in \mathbb S^{n-1}$.

For a signed measure $\nu$ on $\mathbb S^{n-1}$, we define its Fourier transform $\widehat{\nu}$ as 
\[
\widehat{\nu}(\xi)=\int_{\mathbb S^{n-1}} e^{-2\pi\,i\, \xi\cdot v}\,d\nu(v)
\]
for $\xi \in \mathR^n$ (see, for instance, \cite[Chapter II]{Stein70}). Since $\nu$ is compactly supported, $\widehat{\nu}$ is an entire function of exponential type; in particular, $\widehat{\nu}$ is of class $C^{\infty}$ on $\mathbb{R}^{n}$ and all its derivatives are bounded. 
% If we choose a signed measure $\nu \in \calP(\mathS^{n-1})$ satisfying the moment conditions \eqref{condnu}, then we have
% \[
% \widehat{\nu}(0)=0 \quad \text{and} \quad  \nabla\widehat{\nu}(0)=(-2\pi\, i)\,\int_{S^{n-1}} v\,d\nu(v)=0 .
% % \tag{4}
% \]

Now, inserting \eqref{eq:rep} into the definition of $J_p(\nu)$ and using the Fubini-Tonelli theorem, we get
\begin{align}\label{alli}
    &J_p(\nu) 
    =\iint_{\mathS^{n-1}\times \mathS^{n-1}} \!\!\! \left( K(n,p)  \int_{\mathbb{R}^{n}} \frac{1-\cos\bigl(\xi\cdot(v-w)\bigr)-\frac12\bigl(\xi\cdot(v-w)\bigr)^{2}}{|\xi|^{n+p}}\,d\xi \,\right) \, d\nu(v) \, d\nu(w) \\
    &=  \int_{\mathbb{R}^{n}}\frac{K(n,p)}{|\xi|^{n+p}} \left[\iint_{\mathS^{n-1}\times \mathS^{n-1}} \bigl(1-\cos(\xi\cdot(v-w))-\tfrac12(\xi\cdot(v-w))^{2}\bigr) \, d\nu(v)\, d\nu(w)\right]\, d\xi \nonumber \\
    &= \frac{1}{(2\pi)^{p}}\int_{\mathbb{R}^{n}}\frac{K(n,p)}{|\xi|^{n+p}}\iint_{\mathS^{n-1}\times \mathS^{n-1}}\!\!\! \left(1-\cos(2\pi\, \xi\cdot(v-w))-\frac 12(2\pi\, \xi\cdot(v-w))^{2}\right) \! d\nu(v)\, d\nu(w)d\xi.\nonumber
\end{align}
%Because the triple integral is absolutely convergent (the inner integral converges uniformly for $v,w\in S^{n-1}$), we may interchange the order of integration.  
Recalling condition \eqref{condnu} we get 
%We now compute the inner double integral term by term.
%\begin{itemize}
%\item $\displaystyle\iint 1\,d\nu(v)d\nu(w)=\Bigl(\int d\nu\Bigr)^{2}=0$ by the first %condition in (1).
%\item For the cosine term we use the identity
\begin{align*}
  \iint_{\mathS^{n-1} \times \mathS^{n-1}} \cos\bigl(2\pi\, \xi\cdot(v-w)\bigr)d\nu(v)d\nu(w)
  &=\Re\Bigl(\iint_{\mathS^{n-1} \times \mathS^{n-1}} e^{2\pi\,i\,\xi\cdot(v-w)}d\nu(v)d\nu(w)\Bigr) \nonumber\\
  &=\Re\bigl(\widehat{\nu}(\xi) \, \overline{\widehat{\nu}(\xi)}\bigr)
  =|\widehat{\nu}(\xi)|^{2},
\end{align*}
where $\Re(z)$ denotes the real part of $z$, and
%\item For the quadratic term we expand
%  \[
%  (\xi\cdot(v-w))^{2}=(\xi\cdot v)^{2}+(\xi\cdot w)^{2}-2(\xi\cdot v)(\xi\cdot w).
%  \]
%  Hence, using $\int v\,d\nu(v)=0$,
\begin{align*}
  \iint_{\mathS^{n-1} \times \mathS^{n-1}} (2\pi\xi\cdot(v-w))^2 \, d\nu(v) \, d\nu(w) &= 2 \int_{\mathS^{n-1}} (2\pi\xi\cdot v)^{2} \, d\nu(v)\,  \int_{\mathS^{n-1}}\, d\nu(w)\\
  &\quad - 2 \left(\int_{\mathS^{n-1}} 2\pi\xi\cdot v\,d\nu(v) \right)^{2} =0. 
  %&= 2\int_{\mathS^{n-1}} (2\pi\xi\cdot v)^{2} \, d\nu(v),
\end{align*}
Plugging these identities into \eqref{alli}, we eventually obtain
\begin{equation}\label{eqJp}
    J_p(\nu) = -\,\frac{K(n,p)}{(2\pi)^p} \, \int_{\mathR^n} \frac{|\widehat{\nu}(\xi)|^{2}}{|\xi|^{n+p}} 
    %\left( -|\widehat{\nu}(\xi)|^{2}-\frac12\cdot2 \, \int_{\mathS^{n-1}}(2\pi\,\xi\cdot v)^{2}\, d\nu(v) \right)
    \, d\xi .
\end{equation}
Since $K(n,p) < 0$ and the integrand is non‑negative, we conclude that $J_p(\nu) \ge 0$.

If $J_p(\nu)=0$, then $|\widehat{\nu}(\xi)|^{2}=0$ for a.e. $\xi \in \mathR^n$, so that $\widehat{\nu}=0$, namely, $\nu=0$. 
%This completes the proof.
\end{proof}

\begin{proposition}\label{thm:p=4}
    For $p=4$, we have
    \[
        c(n,4)=4+\frac{4}{n},
    \]
    and a measure $\mu\in\mathcal{P}_0(\mathS^{n-1})$ minimizes $J_4$ if and only if it satisfies the isotropy condition
    \begin{equation}\label{isotropy}
        \iint_{\mathS^{n-1} \times \mathS^{n-1}} v_i \, w_j \,d\mu(v)\, d\mu(w)=\frac{1}{n}\delta_{ij}, \qquad i,j=1,\dots,n.
    \end{equation}
    In particular, both $\sigma_{n-1}$ and $\mu_{\rm sim}$ are minimizers.
\end{proposition}

\begin{proof}
    We compute
    \[
        |v-w|^4 = (2-2v\cdot w)^2 = 4(1-v\cdot w)^2 = 4(1-2v\cdot w+(v\cdot w)^2).
    \]
    Integrating and using $\int_{\mathS^{n-1}} v\,d\mu(v)=0$ gives
    \[
        J_4(\mu)= 4 \left(1+ \iint (v\cdot w)^2 \, d\mu(v)\, d\mu(w)\right).
    \]
    Let now $M$ be the matrix with entries $M_{ij} =\int v_i v_j\,d\mu(v)$. 
    Then \[
    \iint (v\cdot w)^2\,d\mu(v)d\mu(w) = \sum_{i,j} (M_{ij})_{1\leq i, \,j \leq n}^2 = \mathrm{tr}(M^2).
    \] 
    Note that $M$ is symmetric, positive semidefinite, and $\mathrm{tr}(M)=\int_{\mathS^{n-1}} |v|^2\,d\mu(v)=1$. By the Cauchy–Schwarz inequality, we have
    \begin{equation}\label{emmedue}
        \mathrm{tr}(M^2) \ge \frac{1}{n} (\mathrm{tr}(M))^2 = \frac{1}{n},
    \end{equation}
    with equality if and only if $M=\frac{1}{n}I$. Thus
    \[
        J_4(\mu) = 4\left(1+ \mathrm{tr}(M^2)\right) \ge 4\left(1+\frac{1}{n}\right),
    \]
    with equality if and only if $M=\frac{1}{n}I$, i.e.,
    \[
        \int_{\mathS^{n-1}} v_i \, v_j\,d\mu(v) = \frac{1}{n}\delta_{ij}.
    \]
    %This condition implies isotropy of the second moments. 
    One can check that, due to symmetry,
    for $\sigma_{n-1}$ and $\mu_{\rm sim}$ it holds $M=\frac{1}{n}I$, thus both measures are minimizers.
\end{proof}

For $p>4$, the minimization problem becomes more delicate.
We first state the following auxiliary result.

\begin{lemma}\label{lem:two-point}
Let $p>4$, $A\in (0,1]$ and $f(t):=(1-t)^{\frac{p}{2}}$. Then every minimizer $\nu$ of
\[
F(A):=\min \left\{ \int_{-1}^1 f(t)\,d\nu(t) :\, \nu\in\mathcal{P}([-1,1]),\;
\int t\,d\nu=0,\; \int t^2\,d\nu=A \right\}
\]
is supported on exactly two points.
\end{lemma}

% By the Krein--Milman theorem, we have that that $\calA = \mathrm{co}(\mathrm{ex}(\calA))$ where $\mathrm{co}(E)$ is the closure of a set $E$ and $\mathrm{ex}(E)$ is the set of extreme points of $E$.

\begin{proof}
Recall 
that the extreme points of $\calP([-1,\,1])$ are Dirac measures on $[-1,\,1]$ (see for instance \cite[Proposition 10.1.3]{Edwards95}), and define
\[
\calA:=\left\{ \nu\in\mathcal{P}([-1,1]):\, \int t\,d\nu=0,\; \int t^2\,d\nu=A \right\}.
\]
Since $\calA$ is the intersection of $\mathcal{P}([-1,1])$ with two closed hyperplanes,
by \cite[Main Theorem]{Dubins62} (see also \cite{Winkler88})
we have that any extreme point of $\calA$ is supported on at most three extreme points of $\mathcal{P}([-1,1])$, that is, it is supported in at most three distinct points.
Moreover, by linearity of the energy in $F(A)$, the minimum is attained on the extreme points of $\calA$.

We now show that a minimizer of the energy that is also an extreme point of $\calA$ has exactly two distinct points in its support. First of all, it is easy to see that such  minimizer cannot be a measure supported only on one point because of the constraint. We now suppose, by contradiction, that a minimizer $\nu$ is supported on three distinct points $t_1 < t_2 < t_3$ with weights $p_1, p_2, p_3 > 0$ such that $p_1 + p_2 + p_3 = 1$. Consider the function
\[
\mathcal{L}(\nu, \lambda_0, \lambda_1, \lambda_2) = \int_{-1}^1 f(t)\,d\nu(t) 
- \lambda_0\left(\int_{-1}^1 d\nu(t) - 1\right) 
- \lambda_1\int_{-1}^1 t\,d\nu(t) 
- \lambda_2\left(\int_{-1}^1 t^2\,d\nu(t) - A\right),
\]
where $\lambda_0, \lambda_1, \lambda_2 \in \mathbb{R}$ are Lagrange multipliers.
For the measure $\nu = \sum_{i=1}^3 p_i\delta_{t_i}$, $\mathcal{L}$ becomes a function of the parameters $\{p_i, t_i\}$
\[
L(p_1,p_2,p_3,t_1,t_2,t_3) = \sum_{i=1}^3 p_i f(t_i) 
- \lambda_0\left(\sum_{i=1}^3 p_i - 1\right) 
- \lambda_1\sum_{i=1}^3 p_i t_i 
- \lambda_2\left(\sum_{i=1}^3 p_i t_i^2 - A\right).
\]
At an extreme point, the gradient of $L$ with respect to all free variables must vanish. This yields
\[
\frac{\partial L}{\partial p_i} = f(t_i) - \lambda_0 - \lambda_1 t_i - \lambda_2 t_i^2 = 0, \qquad i=1,2,3.
\]
Thus for each support point $t_i$ we have
\begin{equation}\label{eq:support}
f(t_i) = \lambda_0 + \lambda_1 t_i + \lambda_2 t_i^2. 
\end{equation}
Since $t_2\in (-1,1)$, we also have
\[
\frac{\partial L}{\partial t_2} = p_2 f'(t_2) - \lambda_1 p_2- 2\lambda_2 p_2 t_2 = 0.
\]
Since $p_2 > 0$, we obtain
\begin{equation}\label{eq:derivative}
f'(t_2) = \lambda_1 + 2\lambda_2 t_2. 
\end{equation}
Define the quadratic polynomial
\[
Q(t) := \lambda_0 + \lambda_1 t + \lambda_2 t^2.
\]
Equation \eqref{eq:support} implies that $f(t_i) = Q(t_i)$ for $i=1,2,3$, and Equation \eqref{eq:derivative} implies that $f'(t_2) = Q'(t_2)$. 
Consequently, the difference function
\[
g(t) := f(t) - Q(t)
\]
satisfies
\[
g(t_i) = 0 \quad\text{for } i=1,2,3\quad \text{and}\quad g'(t_2) = 0.
\]
%Thus each $t_i$ is a zero of $g$ of multiplicity at least two.
Since $g$ is smooth on $[-1,1]$, Rolle's Theorem applied to $g$ on the intervals $[t_1, t_2]$ and $[t_2, t_3]$ yields points $\eta_1 \in (t_1, t_2)$ and $\eta_2 \in (t_2, t_3)$ such that
\[
g'(\eta_1) =  g'(\eta_2) = 0.
\]
Applying again Rolle's Theorem to the function $g'$ on the intervals $[\eta_1, t_2]$ and $[t_2, \eta_2]$, we find points $\xi_1 \in (\eta_1, t_2)$ and $\xi_2 \in (t_2, \eta_2)$ such that
\[
g''(\xi_1) = g''(\xi_2) = 0.
\]
However, we have
\[
g''(t) = f''(t) - 2\lambda_2.
\]
and
\[
f''(t) = \frac{p}{2}\left(\frac{p}{2}-1\right)(1-t)^{\frac{p}{2}-2}
\]
for any $t \in [-1,\,1]$. Since $\frac{p}{2}-2 > 0$, the factor $(1-t)^{\frac{p}{2}-2}$ is strictly decreasing on $[-1,1]$. Moreover, $\frac{p}{2}(\frac{p}{2}-1) > 0$, hence $f''(t)$ itself is strictly decreasing on $[-1,1]$. Consequently, the equation $g''(t) = f''(t) - 2\lambda_2 = 0$ can have at most one solution, contradicting the existence of two distinct points $\xi_1 \neq \xi_2$ with $g''(\xi_1)=g''(\xi_2)=0$.
Therefore, a minimizer cannot be supported in three distinct points.

A completely analogous argument, by restricting to three points contained in the support,
excludes that a minimizer is a discrete measure supported in more than two points.

Assume now that there exists a minimizer $\nu^*$ which is not an extreme point of $\calA$. 
Then, by Choquet Theorem (see for instance \cite[Theorem 10.1.7]{Edwards95}) 
we can write
\[
\nu^* = \int_{\mathcal E} e\,d\mu(e),
\]
where $\mu$ is a probability measure on the set $\mathcal E$ of the extreme points of $\calA$. If $\nu^*$ is not an extreme point, then 
the support of $\mu$ contains at least two elements $\nu_1$, $\nu_2$, which are also minimizers (by linearity of the energy). As a consequence, the discrete measure $\frac{\nu_1+\nu_2}{2}$ would be a minimizer supported on three or four points, leading to a contradiction. It follows that every minimizer is an extreme point of $\calA$, and it is supported on exactly two points. 
\end{proof}

\begin{theorem}\label{thm:main}
For $p > 4$ 
the minimum of $J_p$ is attained uniquely by the measure $\mu_{\mathrm{sim}}$.
In particular,
\[
c(n,p) = 2^{\frac{p}{2}}\left(1+\frac{1}{n}\right)^{\frac{p}{2}-1}.
\]
\end{theorem}

\begin{proof}
For $\mu \in \mathcal{P}_0(\mathbb{S}^{n-1})$ consider the push-forward measure $\nu$ of $\mu\otimes\mu$ under the map $(v,w)\mapsto v\cdot w$. Then $\nu$ is a probability measure on $[-1,1]$ satisfying
\[
\int_{-1}^1 t \, d\nu(t)=0, \qquad \int_{-1}^1 t^2 \, d\nu(t)=A,
\]
where
\[
A:=\iint_{\mathbb{S}^{n-1}\times\mathbb{S}^{n-1}} (v\cdot w)^2 \, d\mu(v)\,d\mu(w)\in (0,1].
\]
Since $|v-w|^p = (2-2v\cdot w)^{\frac{p}{2}}=2^{\frac{p}{2}}(1-v\cdot w)^{\frac{p}{2}}$, we have
\begin{equation}\label{eq:Jp-nu}
J_p(\mu)=2^{\frac{p}{2}}\int_{-1}^1 (1-t)^{\frac{p}{2}}\,d\nu(t).
\end{equation}
For $A\in(0,1]$ we consider the minimum problem
\[
F(A):=\min \left\{ \int_{-1}^1 f(t)\,d\nu(t) \,  : \, \nu\in\mathcal{P}([-1,1]),\;
\int t\,d\nu=0,\; \int t^2\,d\nu=A \right\},
\]
where $f(t):=(1-t)^{\frac{p}{2}}$.
By Lemma~\ref{lem:two-point} we can restrict to measures of the form $\nu=p\delta_a+(1-p)\delta_b$ with $a,b\in[-1,1]$, $\lambda\in(0,1)$. The moment conditions give
\begin{equation}\label{eq:moments}
    \lambda \, a+(1-\lambda)\, b=0 \quad \text{and} \quad  \lambda\, a^2+(1-\lambda)\, b^2=A.
\end{equation}
From the condition $\lambda\, a+(1-\lambda)\, b=0$, we get $b=-\dfrac{\lambda}{1-\lambda}\,a$, hence
\[
    a^2=\frac{A\, (1-\lambda)}{\lambda},\quad b=-\frac{A}{a}, \quad \text{and} \quad \lambda = \frac{A}{A+a^2}.
\]
By symmetry we can assume $a\ge0$, then the constraints $a\le1$, $b\ge-1$ imply $a\in[A,1]$. 
Define
\[
    G(a):=\frac{A}{A+a^2}\, f(a)+\frac{a^2}{A+a^2}\, f\left(-\frac{A}{a}\right).
\]
We claim that $\min_{a\in[A,1]} G(a)=G(1)$. Indeed, we shall prove that \(G'(a) < 0\) for every \(a \in [A,1)\), hence \(G\) is strictly decreasing on $[A,1]$.
Set \(\alpha = \frac{p}{2} > 2\). Then
\[
    f(t) = (1-t)^\alpha \quad \text{and} \quad f'(t) = -\alpha\, (1-t)^{\alpha-1}.
\]
Define \(D = A+a^2\). The derivative is
\[
    G'(a)=\frac{2aA}{D^2}\, \left[f\left(\frac{-A}{a}\right)-f(a) \right] + \frac{A}{D}\, \left[f'(a)+ f'\left(\frac{-A}{a}\right) \right].
\]
Since \(A>0\) and \(D>0\), the sign of \(G'(a)\) equals the sign of
\[
    K(a):=\frac{2a}{D}\, \left[f\left(\frac{-A}{a}\right)-f(a) \right] + f'(a) + f'\left(\frac{-A}{a}\right).
\]
Introduce the variables
\[
    x=1-a\quad \text{and} \quad y=1+\frac{A}{a}.
\]
Then \(0 < x \le 1-A\) and \(y \ge 1+A\). Moreover,
\[
    a=1-x,\quad \frac{A}{a}=y-1,\quad  \text{and} \quad A=a(y-1)=(1-x)(y-1),
\]
and thus we have
\[
    D=A+a^2=(1-x)(y-1)+(1-x)^2=(1-x)(y-x).
\]
Hence
\[
    \frac{2a}{D}=\frac{2(1-x)}{(1-x)(y-x)}=\frac{2}{y-x}.
\]
We also have
\[
    f\!\left(-\frac{A}{a}\right)=\left(1+ \frac{A}{a} \right)^{\alpha} = y^{\alpha}, 
    \quad \ f(a)=(1-a)^\alpha=x^\alpha,
\]
and 
\[
f'\left( \frac{-A}{a} \right)=-\alpha \, \left(1+ \frac{A}{a}\right)^{\alpha-1} =-\alpha y^{\alpha-1}, \quad \ 
f'(a)=-\alpha \, (1-a)^{\alpha-1}=-\alpha x^{\alpha-1}.
\]
Therefore,
\[
K(a)=\frac{2}{y-x}(y^\alpha-x^\alpha)-\alpha\bigl(x^{\alpha-1}+y^{\alpha-1}\bigr).
\]
Multiplying by the positive quantity \(y-x\) gives
\[
(y-x)K(a)=2(y^\alpha-x^\alpha)-\alpha(y-x)\bigl(x^{\alpha-1}+y^{\alpha-1}\bigr).
\]
Thus \(K(a)<0\) if and only if
\begin{equation}\label{stella}
2(y^\alpha-x^\alpha)<\alpha(y-x)\bigl(x^{\alpha-1}+y^{\alpha-1}\bigr). 
\end{equation}
Set \(t=\frac{x}{y} \, (0<t<1)\). Then \(y-x=y(1-t)\) and after division by \(y^\alpha\), the inequality \eqref{stella} becomes
\begin{equation}\label{stellastella}
2(1-t^\alpha)<\alpha(1-t)\bigl(t^{\alpha-1}+1\bigr). 
\end{equation}
Define \(\psi(t)=\alpha(1-t)\bigl(1+t^{\alpha-1}\bigr)-2(1-t^\alpha)\) for \(t\in(0,1]\). We shall prove \(\psi(t)>0\) for all \(t\in(0,1)\). Since \(\psi(1)=0\), it suffices to show that \(\psi\) is strictly decreasing on \((0,1)\). Compute
\[
\psi'(t)=\alpha\Bigl[- \bigl(1+t^{\alpha-1}\bigr)+(1-t)(\alpha-1)t^{\alpha-2}\Bigr]+2\alpha t^{\alpha-1}
        =\alpha\Bigl[-1+t^{\alpha-1}+(\alpha-1)(1-t)t^{\alpha-2}\Bigr].
\]
We rewrite the bracket as
\[
-1+t^{\alpha-2}\bigl[t+(\alpha-1)(1-t)\bigr]
   =-1+t^{\alpha-2}\bigl[(\alpha-1)-(\alpha-2)t\bigr].
\]
Set \(g(t)=t^{\alpha-2}\bigl((\alpha-1)-(\alpha-2)t\bigr)\). Then \(\psi'(t)=\alpha\bigl(-1+g(t)\bigr)\). The derivative of \(g\) is
\begin{eqnarray*}
g'(t)&=&(\alpha-2)t^{\alpha-3}\bigl((\alpha-1)-(\alpha-2)t\bigr)+t^{\alpha-2}\bigl(-(\alpha-2)\bigr)\\
      &=&(\alpha-2)t^{\alpha-3}\bigl[(\alpha-1)-(\alpha-2)t-t\bigr]\\
      &=&(\alpha-2)(\alpha-1)t^{\alpha-3}(1-t).
\end{eqnarray*}
For \(\alpha>2\) we have \((\alpha-2)(\alpha-1)>0\), and for \(t\in(0,1)\) also \(t^{\alpha-3}>0\) and \(1-t>0\); hence \(g'(t)>0\). Thus \(g\) is strictly increasing on \((0,1]\). Since \(g(1)=1\), it follows that \(g(t)<1\) for every \(t\in(0,1)\). Consequently \(\psi'(t)=\alpha(-1+g(t))<0\) for all \(t\in(0,1)\), so \(\psi\) is strictly decreasing. Because \(\psi(1)=0\), we obtain \(\psi(t)>0\) for every \(t\in(0,1)\), which is exactly inequality \eqref{stellastella}.
Therefore \eqref{stella} holds, hence \(K(a)<0\) and finally \(G'(a)<0\) for every \(a\in[A,1)\). Thus \(G\) is strictly decreasing on \([A,1]\), and its minimum is attained at the right endpoint \(a=1\). It follows that
\[
F(A)=\min_{a\in[A,1]}G(a)=G(1)=(1+A)^{\alpha-1},
\]
and the minimizer corresponds to $a=1$, $b=-A$, and $p= \frac{A}{1+A}$, i.e.,
\begin{equation}\label{eq:optimal-nu}
\nu_A=\frac{A}{1+A}\,\delta_1+\frac{1}{1+A}\,\delta_{-A}.
\end{equation}
From \eqref{eq:Jp-nu} and the definition of $F(A)$, we get
\begin{equation}\label{oibo}
J_p(\mu)=2^{\frac{p}{2}}\, \int_{-1}^{1} f(t) \,d\nu(t) \ge 2^{\frac{p}{2}}\, F(A) = 2^{\frac{p}{2}}(1+A)^{\frac{p}{2}-1},
\end{equation}
for all measures $\mu\in \mathcal{P}_0(\mathbb{S}^{n-1})$.

Letting now $M$ be such that
$M_{ij} = \int_{\mathS^{n-1}} v_i v_j \, d\mu(v)$,
%Since \(\int v\,d\mu(v)=0\) and \(\|v\|=1\), we have \(\operatorname{tr}(M_2) = \int \|v\|^2 d\mu = 1\), and 
%\[
%A = \iint (v\cdot w)^2 d\mu(v)d\mu(w)
%       = \sum_{i,j} \Bigl(\int v_i v_j d\mu\Bigr)\Bigl(\int w_i w_j d\mu\Bigr)
%       = \sum_{i,j} (M_2)_{ij}^2 = \operatorname{tr}(M_2^2).
%\]
by \eqref{emmedue} we have
\[
A = \operatorname{tr}(M^2) \ge \frac{1}{n},
\]
with the equality if and only if \(M = \frac{1}{n}I_n\) (i.e. \(\mu\) is isotropic). 
Since $A\mapsto(1+A)^{\frac{p}{2}-1}$ is increasing for $p>4$, from \eqref{oibo} we then obtain
\begin{equation}\label{eq:lower-bound}
J_p(\mu)\ge 2^{\frac{p}{2}}\Bigl(1+\frac{1}{n}\Bigr)^{\frac{p}{2}-1}.
\end{equation}
The equality forces $A= \frac{1}{n}$ and $\nu=\nu_{\frac{1}{n}} = \frac{1}{n+1}\,\delta_1+\frac{n}{n+1}\,\delta_{-\frac{1}{n}}$.
Thus for $(\mu\otimes\mu)$-almost every $(v,w) \in \mathS^{n-1} \times \mathS^{n-1}$, $v \cdot w\in\{1,-\frac{1}{n}\}$. 
Hence the support of $\mu$ consists of unit vectors with pairwise inner products $-\frac{1}{n}$, which are the vertices of a regular $n$-simplex. The barycenter condition $\int_{\mathS^{n-1}} v\,d\mu(v)=0$ forces $\mu$ to be uniform on these $n+1$ vertices, i.e., $\mu=\mu_{\mathrm{sim}}$.

Any other minimizer $\mu$ induces the same measure $\nu$, hence 
for \(\mu\otimes\mu\)-almost every \((v,w)\) we have \(v\cdot w \in \{1, -\frac{1}{n}\}\),
so that
\[
    |v-w|^2 = 2-2v\cdot w \in \left\{0,\; 2\, \left(1+\frac{1}{n}\right) \right\}.
\]
Thus, distinct points in the support of \(\mu\) are at constant distance \(\sqrt{2(1+\frac{1}{n})}\); equivalently, they form a set of unit vectors with pairwise inner product \(-\frac{1}{n}\). Such a configuration is, up to rotation, exactly the set of vertices of a regular \(n\)-simplex. Moreover, the barycenter constraint forces the measure to be uniform on those vertices, so that the minimum is uniquely attained by $\mu_{\mathrm{sim}}$.
\end{proof}

\section{Optimal sets for \texorpdfstring{$\calE_p$}{Ep}}
Building on the results of the previous section, we now 
characterize minimizers and maximizers for the energy $\calE_p$ under perimeter or volume constraint.

\subsection{Perimeter constraint}\label{sec:perimeter}
Given $P_0>0$, we consider the problems
\begin{align}
&\min\{\calE_p(\Omega) :\, \Omega\subset\mathR^n,\, \Per(\Omega)=P_0\}, \label{prob:per-min}\\
&\max\{\calE_p(\Omega) :\, \Omega\subset\mathR^n,\, \Per(\Omega)=P_0\}. \label{prob:per-max}
\end{align}

%\begin{theorem}\label{thm:exist-per}
%For every $p>0$ and every $P_0>0$, problems \eqref{prob:per-min} and \eqref{prob:per-max} admit solutions.
%\end{theorem}
%\begin{proof}
%For minimization, let $\{\Omega_k\}$ be a minimizing sequence with $\Per(\Omega_k)=P_0$.
%By the compactness theorem for sets of finite perimeter, after extracting a subsequence we may assume $\Omega_k\to\Omega$ in $L^1_{\rm loc}$ and $\Per(\Omega) \le \liminf_{k \to +\infty} \Per(\Omega_k)=P_0$. The perimeter may drop in the limit, but we can always rescale: if $\Per(\Omega)<P_0$, let $\widetilde\Omega=\lambda\Omega$ with $\lambda=(P_0/\Per(\Omega))^{1/(n-1)}$. Then $\Per(\widetilde\Omega)=P_0$ and by the scaling property $\calE_p(\widetilde\Omega)=\lambda^{2(n-1)}\calE_p(\Omega)$. A lower semicontinuity argument shows that $\widetilde\Omega$ is a minimizer.

%For maximization, note that $\calE_p(\Omega)\le 2^p \Per(\Omega)^2$ because $|v-w|^p\le 2^p$. Thus the maximum is finite. Let $\{\Omega_k\}$ be a maximizing sequence. Again by compactness, we can assume $\Omega_k\to\Omega$ in $L^1_{\rm loc}$ and $\Per(\Omega)\le P_0$. If $\Per(\Omega)<P_0$, rescaling as above would increase $\calE_p$ (since $\lambda>1$ and $p>0$), contradicting maximality. Hence $\Per(\Omega)=P_0$ and by lower semicontinuity $\calE_p(\Omega)\ge \limsup \calE_p(\Omega_k)$, so $\Omega$ is a maximizer.
%\end{proof}

We partially characterize the optimal sets for different ranges of $p$.

\begin{theorem}\label{thm:p<2-per}
For $0<p<2$:
\begin{enumerate}
\item The infimum of \eqref{prob:per-min} is not attained.
A minimizing sequence is given by thin cylinders with the heights tending to zero.
\item The ball of perimeter $P_0$ is a maximizer of \eqref{prob:per-max}, and it is the unique maximizer among convex sets.
\end{enumerate}
\end{theorem}

\begin{proof}
We first prove $(1)$. From Proposition \ref{thm:p<2-min}, we have
\begin{equation*}
    \inf \{\mathrm{E}_p(\Omega) : \Omega \subset \mathR^n ,\, |\Omega| < +\infty, \, \Per(\Omega) = P_0\} \geq P_0^2\, c(n,p) = P_0^2\, 2^{p-1},
\end{equation*}
and $J_p(\mu)=c(n,p)$ is achieved only by measures of the form $\mu = \frac{1}{2}(\delta_{v} + \delta_{-v})$ for some $v \in \mathS^{n-1}$. 
Such measures are not induced by a set of finite perimeter, so the the infimum is not attained. To see this, we construct a minimizing sequence such whose energy converges to $P_0^2\,2^{p-1}$. We set $\Omega_{\varepsilon} := [0,\,\varepsilon]\times[0,\, L_\varepsilon]^{n-1}$ where $L_{\varepsilon} > 0$ satisfying $\lim_{\varepsilon \downarrow 0}L^{n-1}_{\varepsilon} = \frac{P_0}{2}$, so that $\Per(\Omega_{\varepsilon}) = P_0$. Then, an easy computation gives 
\begin{align*}
    \mathrm{E}_p(\Omega_{\varepsilon}) &= 2\, \iint_{F^1_{\varepsilon} \times F^2_{\varepsilon}} 2^{\frac{p}{2}}\, \left( 1 - \nu_{\Omega_{\varepsilon}}(x) \cdot \nu_{\Omega_{\varepsilon}}(y) \right)^{\frac{p}{2}}\,d\calH^{n-1}\,d\calH^{n-1} + \mathrm{O}(\varepsilon) \nonumber\\
    &=  2^{p+1}\, \calH^{n-1}(F^1_{\varepsilon}) \, \calH^{n-1}(F^2_{\varepsilon})  + \mathrm{O}(\varepsilon),
\end{align*}
where $F^1_{\varepsilon} \coloneqq \{0\} \times [0,\,L_{\varepsilon}]^{n-1}$ and $F^2_{\varepsilon} \coloneqq \{\varepsilon\} \times [0,\,L_{\varepsilon}]^{n-1}$.
%From Proposition~\ref{thm:p<2-min}, the minimum of $J_p$ is $2^{p-1}$, attained only by $\mu_{\rm slab}$. A set inducing $\mu_{\rm slab}$ must have a boundary consisting of two parallel hyperplanes, i.e., a slab. For bounded sets, we can approximate it by $\Omega_\varepsilon=[0,\varepsilon]\times[0,L_\varepsilon]^{n-1}$ with appropriate $L_\varepsilon$ having a perimeter $P_0$. Then 
Letting $\varepsilon \downarrow 0$, we have
\begin{equation*}
    \lim_{\varepsilon \downarrow 0} \calE_p(\Omega_\varepsilon)   = 2^{p-1} \, P_0^2.
\end{equation*}
This completes the proof of the first claim.

Regarding $(2)$, by Proposition \ref{thm:p<2-min}, the maximum of $J_p(\mu)$ is uniquely attained at $J_p(\sigma_{n-1})$, which corresponds to a ball of perimeter $P_0$. 
\end{proof}

\begin{remark}
By Proposition~\ref{prop:p=2} we have $\calE_2(\Omega)=2\Per(\Omega)^2$,
thus every set with perimeter $P_0$ has the same energy. 
%there is no nontrivial optimization.
\end{remark}

\begin{theorem}\label{thm:2<p<4-per}
For $2<p<4$:
\begin{enumerate}
\item The ball of perimeter $P_0$ is a minimizer of \eqref{prob:per-min},
and it is the unique minimizer among convex sets.
\item The supremum in \eqref{prob:per-max} is not attained.
A maximizing sequence is given by thin cylinders with the heights tending to zero.
\end{enumerate}
\end{theorem}

\begin{proof}
By Proposition~\ref{thm:2<p<4} the unique minimizer of $J_p$ is $\sigma_{n-1}$,
and the maximizers are measures of the form $\mu = \frac{1}{2}(\delta_{v} + \delta_{-v})$ for some $v \in \mathS^{n-1}$. 
The thesis then follows as in the proof of Theorem \ref{thm:p<2-per}.
\end{proof}

\begin{theorem}\label{thm:p=4-per}
For $p\ge 4$:
\begin{enumerate}
\item For $p=4$, both the ball and the regular simplex of perimeter $P_0$ are minimizers of \eqref{prob:per-min}. 
\item For $p>4$, the regular simplex of perimeter $P_0$ is a minimizer of \eqref{prob:per-min}, and it is the unique minimizer among convex sets.
\item For $p\ge 4$, the supremum of \eqref{prob:per-max} is not attained.
A maximizing sequence is given by thin cylinders with the heights tending to zero.
\end{enumerate}
\end{theorem}

\begin{proof}
By Proposition~\ref{thm:p=4}, any measure satisfying the isotropy condition \eqref{isotropy} minimizes $J_4$. The ball induces $\sigma_{n-1}$ and the simplex induces $\mu_{\rm sim}$, which are both isotropic measures.

By Theorem \ref{thm:main}, the minimum of $J_p$ is attained uniquely by $\mu_{\rm sim}$,
which is induced by the regular simplex.

By Propositions~\ref{thm:p=4} and \ref{max:2<p}, the maximum of $J_p$ is attained by measures of the form $\mu = \frac{1}{2}(\delta_{v} + \delta_{-v})$ for some $v \in \mathS^{n-1}$, 
and the thesis follows as in the proof of Theorem \ref{thm:p<2-per}.
\end{proof}

%\begin{remark}
%    Thanks to Minkowski's inequality in \cite[Theorem 7.2.1]{S14}, we observe that the following minimization problem
%    \begin{equation*}
%        \min\{ \mathrm{E}_p(C) : \, \text{$C \subset \mathR^n$ is convex with $\Per(C) = P_0$} \}
%    \end{equation*}
%    has a unique solution for any $p>0$ with $p \neq 2,\,4$. Indeed, 
%\end{remark}

\subsection{Volume Constraint}\label{sec:volume}
Given $V_0>0$, we now consider the minimum problem
\begin{equation}\label{prob:vol-min}
\min\{ \mathrm{E}_p(\Omega) :\ \Omega\subset\mathR^n,\, |\Omega|=V_0\}. 
\end{equation}

\begin{theorem}\label{thm:exist-vol}
For every $p>0$ problem \eqref{prob:vol-min} admits a minimizer,
which is bounded and convex.
\end{theorem}

\begin{proof}
Let $\{\Omega_k\}$ be a minimizing sequence for problem \eqref{prob:vol-min}.
From Remark \ref{rmk:applicationMinkowski} and Proposition \ref{provol}), we can assume that $\Omega_k$ are convex sets. Indeed, if $C_k$ are bounded convex sets such that 
$\Per(C_k)=\Per(\Om_k)$, $|C_k|\ge |\Om_k|=V_0$ and $\mathrm{E}_p(C_k)=\mathrm{E}_p(\Om_k)$,
letting $\widetilde \Om_k= (V_0/|C_k|)^{\frac 1n}C_k$ we have $\mathrm{E}_p(\widetilde \Om_k)\le\mathrm{E}_p(\Om_k)$, with equality if and only if $\Om_k$ is convex.

Letting $B^{V_0}$ be the ball of volume $V_0$, 
by Proposition \ref{prop:basic} we have
\[
\Per(\Om_k)^2\le \frac{\mathrm{E}_p(\Om_k)}{c(n,p)}\le \frac{\mathrm{E}_p(B^{V_0})}{\min(2^{\frac{p}{2}},2^{p-1})},
\]
that is, the perimeter of the sets $\Om_k$ is uniformly bounded.
As a consequence, the diameter of $\Om_k$ is also uniformly bounded 
(see for instance \cite{GWW87}),
so that, up to a subsequence, $\chi_{\Om_k} \to \chi_{\Om}$ as $k \to +\infty$ in $L^1$, where $\Om$ is a bounded convex set
of volume $V_0$. Since there holds $\Per(\Om_k)\to \Per(\Om)$, by Theorem \ref{teominkconv}
it follows that $\Om$ is a solution of problem \eqref{prob:vol-min}.
\end{proof}

Since $J_2(\mu)=2$ and $\mathrm{E}_2(\Omega)=2\Per(\Omega)^2$, for $p=2$ the minimization of $\mathrm{E}_p$
reduces to minimizing perimeter at fixed volume, which gives the ball  $B^{V_0}$
as unique minimizer. More generally, we can show that the ball is the unique minimizer also for $2<p\le 4$.

\begin{theorem}\label{thm:2<p<4-vol}
For $2\le p\le 4$ the ball is the unique minimizer of $\mathrm{E}_p$.
\end{theorem}
\begin{proof}
By Propositions \ref{thm:2<p<4} and \ref{thm:p=4}, in this range of $p$
the measure $\sigma_{n-1}$ is a minimizer of $J_p$, hence
\[
\mathrm{E}_p(\Om)\ge \Per(\Om)^2 J_p(\sigma_{n-1})\ge \Per(B^{V_0})^2 J_p(\sigma_{n-1}) 
= \mathrm{E}_p(B^{V_0}),
\]
for every finite perimeter set $\Om$ with volume $V_0$, 
and the equality holds if and only 
if $\Om$ is a ball of volume $V_0$.
\end{proof}

\begin{remark}
    In the case $0<p<2$ the ball minimizes the perimeter but the antipodal measure minimizes $J_p$ (by Proposition \ref{thm:p<2-min}), so we expect that the minimizer is a ball for $p$ close enough to $2$, while it degenerates to a hyperplane of multiplicity $2$ as $p\to 0$.

    In the case $p>4$ the ball minimizes the perimeter but the uniform measure on the regular simplex minimizes $J_p$ (by Theorem \ref{thm:main}), so we expect that the minimizer is a ball for $p$ close enough to $4$, and a regular simplex for $p$ large enough.
\end{remark}

\subsection{Closed curves of fixed length}\label{seccurves}

In this section, we consider the minimization and maximization of a geometric oscillation energy for closed rectifiable curves in \(\mathbb{R}^n\). We set $\calA([0,L]; \,\mathR^n)$ be the collection of rectifiable closed curves of length $L>0$, parametrized by arc-length. For $\gamma \in \calA([0,L]; \mathR^n)$, we denote the unit tangent vector of $\gamma$ at \(\gamma(s)\) by \(\tau(s) \in \mathS^{n-1}\), and we define
\[
\mathrm{E}_p(\gamma) := \iint_{[0,L]^2} |\tau(s) - \tau(t)|^p \, ds \, dt
\]
for $p>0$.
%Note that this functional is scale-invariant: if we dilate the curve, the length changes but the energy scales as \(\mathrm{E}_p(\lambda \gamma) = \lambda^2 \mathrm{E}_p(\gamma)\), while the length becomes \(\lambda L\). Therefore, it is natural to fix the length. 
We consider the problems
\begin{equation}
\min \{ \mathrm{E}_p(\gamma)  :\, \gamma \in \calA([0,L]; \mathR^n)  \} \label{prob:curve-min} 
\end{equation}
and
\begin{equation}
\max \{ \mathrm{E}_p(\gamma) : \, \gamma \in \calA([0,L]; \mathR^n)  \}. \label{prob:curve-max}
\end{equation}
As in the case of hypersurfaces, we can relate this problem to a variational problem on probability measures on the sphere. Define the push-forward measure
\[
\mu_\gamma = \frac{1}{L} \, \tau_\# (\calH^1 \lfloor_{[0,L]}) \in \mathcal{P}_0(S^{n-1}),
\]
and we have
\[
\mathrm{E}_p(\gamma) = L^2 J_p(\mu_\gamma).
\]
%where \(J_p(\mu) = \iint_{S^{n-1} \times S^{n-1}} |v-w|^p \, d\mu(v) \, d\mu(w)\) as before. Moreover, because the curve is closed, we have
%\[
%\int_0^L \tau(s) \, ds = 0,
%\]
%which implies that the barycenter of \(\mu_\gamma\) is zero: \(\int_{S^{n-1}} v \, d\mu_\gamma(v) = 0\). Hence, \(\mu_\gamma \in \mathcal{P}_0(S^{n-1})\).
% \fumihiko{I couldn't find a reference in the original file.}
Conversely, given any \(\mu \in \mathcal{P}_0(\mathS^{n-1})\), there exists a rectifiable curve \(\gamma\) of length \(L\) such that \(\mu_\gamma = \mu\). Indeed, by the Isomorphism theorem for measures (see \cite[Theorem 17.41]{K95}), there exists a measurable map \(u : [0,L] \to \mathS^{n-1}\) such that \(\mu = \frac{1}{L} u_\# \mathcal{L}^1 \lfloor_{[0,L]}\) and \(\int_0^L u(s)\, ds = 0\). Then one can define \(\gamma(s) = \int_0^s u(t)\, dt\), which yields a closed rectifiable curve with tangent field \(u\). Therefore, problems \eqref{prob:curve-min} and \eqref{prob:curve-max} are equivalent to
\[
\min_{\mu \in \mathcal{P}_0(\mathS^{n-1})} L^2 J_p(\mu) \quad \text{and} \quad \max_{\mu \in \mathcal{P}_0(\mathS^{n-1})} L^2 J_p(\mu),
\]
respectively. 
%Since the constant \(L^2\) is irrelevant for the optimization, we are led to study the minima and maxima of \(J_p\) over \(\mathcal{P}_0(S^{n-1})\).

%The results obtained in Section \ref{sec:constants} for the constants \(c(n,p)\) and \(C(n,p)\) provide a complete characterization of these extrema for different ranges of \(p\). We now translate those results into geometric statements about curves.

\begin{theorem}[Minimizing curves]\label{thm:curve-min}
Let \(L > 0\) be fixed. The minimum value of \(\mathrm{E}_p\) among closed rectifiable curves of length \(L\) is \(L^2 c(n,p)\), where \(c(n,p)\) is given by Theorem \ref{thm:p<2-min}, Proposition \ref{prop:p=2}, Theorem \ref{thm:2<p<4}, Theorem \ref{thm:p=4}, and Theorem \ref{thm:main} for the respective ranges of \(p\). In particular,
\begin{enumerate}
\item For \(0 < p < 2\), the minimum is attained by curves that are multiply covered segments of length \(L\) (where the length is counted with multiplicity). The minimum value is \(2^{p-1} L^2\).
\item For \(p=2\), every curve of length \(L\) has energy \(2L^2\) and is a minimizer.
\item For \(2 < p < 4\) minimizers are curves whose tangent measure is the uniform measure \(\sigma_{n-1}\). For $n=2$ the circle of length \(L\) is a minimizer, and it is the unique minimizer among convex curves.
\item For \(p=4\), there are infinitely many minimizers, given by curves whose tangent measure satisfies the isotropy condition \eqref{isotropy}. For $n=2$,
among them there is the circle and all the regular polygons of perimeter \(L\). 
\item For \(p>4\), minimizers are closed curves whose tangent measure is 
\(\mu_{\mathrm{sim}}\). For \(n=2\) equilateral triangles of perimeter \(L\) are minimizers,
and they are the unique minimizers among convex curves.
\end{enumerate}
\end{theorem}
\begin{proof}
The statements follow directly from the characterization of \(c(n,p)\) and the corresponding minimizers of \(J_p\). 

For \(0<p<2\), by Proposition \ref{thm:p<2-min} the minimizer of \(J_p\) is \(\mu = \frac{1}{2}(\delta_v + \delta_{-v})\), which corresponds to curves whose tangent vector takes only the two opposite values \(v\) and \(-v\). Such curves are multiply covered segments in the direction \(v\), with even multiplicity at every point.

The case \(p=2\) is trivial since \(J_2\) is constant.

For \(2<p<4\), Proposition \ref{thm:2<p<4} states that the unique minimizer of \(J_p\) is \(\sigma_{n-1}\), so that minimizing curves must have \(\mu_\gamma = \sigma_{n-1}\). For \(n=2\) this 
condition is satisfied by the circle of length $L$,
for \(n>2\) there exist curves whose tangent vector is uniformly distributed on $\mathS^{n-1}$ (see Remark \ref{remfil} below). 

For \(p=4\) the result follows directly from Proposition \ref{thm:p=4}.

For \(p>4\), Theorem \ref{thm:main} states that the unique minimizer of \(J_p\) is 
\(\mu_{\mathrm{sim}}\), so that minimizing curves must have \(\mu_\gamma = \mu_{\mathrm{sim}}\). For \(n=2\) this 
condition is satisfied by equilateral triangles of perimeter \(L\). 
\end{proof}

\begin{remark}\label{remfil}
For \(2<p<4\) and $n >2$ minimizers are non-planar curves whose tangent vectors cover uniformly the unit sphere, and can be constructed by integrating Peano-type curves with values on \(\mathS^{n-1}\) (see for instance \cite[Chapter 3]{Sagan94}). Notice that these curves can be of class \(C^{1,\alpha}\) with $\alpha\le \frac 12$, but not of class \(C^{1,1}\). 
\end{remark}

In the following theorem we describe the closed curves maximizing \eqref{prob:curve-max}.
We omit the proof which is analogous to that of Theorem \ref{thm:curve-min}.

\begin{theorem}[Maximizing curves]\label{thm:curve-max}
Let \(L > 0\) be fixed. The maximum value of \(\mathrm{E}_p\) among closed rectifiable curves of length \(L\) is \(L^2 \,C(n,p)\), where \(C(n,p)\) is given by Theorem \ref{thm:p<2-min}, Proposition \ref{prop:p=2}, Theorem \ref{thm:2<p<4}, Theorem \ref{thm:p=4}, and Theorem \ref{thm:main}. In particular,
\begin{enumerate}
\item For \(0 < p < 2\), maximizers are curves whose tangent measure is the uniform measure \(\sigma_{n-1}\). For $n=2$ the circle of length \(L\) is a maximizer, and it is the unique maximizer among convex curves. The maximum value is \(L^2 J_p(\sigma_{n-1})\).
\item For \(p=2\), every curve has energy \(2L^2\) and is a maximizer.
\item For \(p > 2\), the maximum is attained by curves that are multiply covered segments of length \(L\) (where the length is counted with multiplicity). The maximum value is \(2^{p-1} L^2\).
\end{enumerate}
\end{theorem}

%%%%%%%%%%%%%%%%%%%%%%%%%%%%%%%%%%%%%%%%
%%%%%%%%%%%%%%%%%%%%%%%%%%%%%%%%%%%%%%%%
%%%%%%%%%%%%%%%%%%%%%%%%%%%%%%%%%%%%%%%% 
%%%%%%%%%%%%%%%%%%%%%%%%%%%%%%%%%%%%%%%%
%%%%%%%%%%%%%%%%%%%%%%%%%%%%%%%%%%%%%%%%
%%%%%%%%%%%%%%%%%%%%%%%%%%%%%%%%%%%%%%%%
\bibliographystyle{plain}
\bibliography{bibliography.bib}
	
\end{document}